\documentclass[12pt,leqno]{article}
\date{}
\usepackage{amsfonts}
\usepackage{bbm}
\usepackage{amsmath}
\usepackage{mathrsfs}
\usepackage{amssymb}
\usepackage{amsmath,color}
\usepackage{amsthm}
\usepackage{amstext}
\usepackage{amsopn}
\usepackage{color}
\usepackage{texdraw}
\usepackage{graphicx}
\usepackage{multirow}
\usepackage{lscape}
\usepackage{float}
\usepackage{txfonts}

\usepackage{subfigure}
\newtheorem{theorem}{Theorem}[section]
\newtheorem{lemma}[theorem]{Lemma}
\newtheorem{corollary}[theorem]{Corollary}

\newtheorem{definition}[theorem]{Definition}
\theoremstyle{plain}
\newtheorem{remark}[theorem]{Remark}
\numberwithin{equation}{section}

\newcommand{\Om}{\Omega}

\newcommand{\ds}{\displaystyle}

\newcommand{\N}{\mathcal{N}}

\newcommand{\R}{\mathbb{R}}

\begin{document}

\title{Improvement of conditions for finite time blow-up in a fourth-order nonlocal parabolic equation}

\author{Jingbo Meng $^a$,\ Shuyan Qiu $^b$,\ Guangyu Xu $^{a,}$\footnote{Corresponding author: guangyuswu@126.com (G. Xu)}, Hong Yi $^{c,}$\footnote{E-mail addresses: meng4897@zjnu.edu.cn (J. Meng), shuyanqiu0701@126.com (S. Qiu), honghongyi1992@126.com (H. Yi)} \\
{\small a. School of Mathematical Sciences, Zhejiang Normal University, Jinhua 321004, P.R. China} \\ \small b. School of Sciences, Southwest Petroleum University, Chengdu 610500, P.R. China \\ \small c. School of Mathematics and Statistics, Chongqing Technology and Business University, \\ \small Chongqing 400067, P.R. China}

\date{}
\maketitle

\begin{abstract}
  This paper is devoted to the study of blow-up phenomenon for a fouth-order nonlocal parabolic equation with Neumann boundary condition,
  \begin{equation*}
  \left\{\begin{array}{ll}\ds u_{t}+u_{xxxx}=|u|^{p-1}u-\frac{1}{a}\int_{0}^a|u|^{p-1}u\ dx, & \\
  u_x(0)=u_x(a)=u_{xxx}(0)=u_{xxx}(a)=0, & \\
  u(x,0)=u_0(x)\in H^2(0, a),\ \ \int_0^au_0(x)\ dx=0, &\end{array}\right.
  \end{equation*}
  where $a$ is a positive constant and $p>1$. The existing results on the problem suggest that the weak solution will blow up in finite time if $I(u_0)<0$ and the initial energy satisfies some appropriate assumptions, here $I(u_0)$ is the initial Nehari functional. In this paper, we extend the previous blow-up conditions with proving that those assumptions on the energy functional are superfluous and only $I(u_0)<0$ is sufficient to ensure the weak solution blowing up in finite time. Our conclusion depicts the significant influence of mass conservation on the dynamic behavior of solution.
\end{abstract}
\maketitle

\section{Introduction}

In this paper, we consider the following initial boundary value problem:
\begin{equation}\label{fc1}
   \left\{\begin{array}{ll}\ds u_{t}+u_{xxxx}=|u|^{p-1}u-\frac{1}{a}\int_{0}^a|u|^{p-1}udx,&(x,t)\in(0,a)\times(0, T),\\
  u_x(0)=u_x(a)=u_{xxx}(0)=u_{xxx}(a)=0,&t\in(0,T),\\
  u(x,0)=u_0(x),&x\in(0,a),\end{array}\right.
\end{equation}
where $a$ is a positive constant, $p>1, T\in(0, +\infty]$ is the maximal existence time of $u(x, t)$, and the initial data satisfies
\begin{equation}\label{csz}
u_0(x)\in H^2(0,a),\ \ \int_0^au_0(x)dx=0, \ \ u_0(x)\not\equiv0.
\end{equation}

Problem \eqref{fc1} can be taken as a simplified model to describe the evolution of the epitaxial growth of nanoscale thin films. The continuum model for epitaxial thin-film growth was proposed in \cite{ors}, which based upon phenomenological considerations by Zangwill \cite{za}. In this setting, for the spatial and time variable $(x, t)$, the unknown function $u=u(x, t)$ depicts the height of a film in epitaxial growth. The nonlocal source $|u|^{p-1}u-\frac{1}{a}\int_0^a|u|^{p-1}udx$ plays an important role in the whole research process, which together with the homogeneous Neumann boundary condition and the assumption \eqref{csz} on initial data imply that the weak solution $u$ of problem \eqref{fc1} satisfies:
\begin{equation}\label{zlsh}
\int_0^a u_t\ dx=0,
\end{equation}
which further implies that $  \int_0^au(x, t)dx=\int_0^au_0(x)dx=0.$ Namely, the total mass is conserved throughout its evolution. For more derivations and background of the model, one can refer to the references \cite{Qu2016Blow,xzm}. We would like to mention the works \cite{Budd1994Blow,Gao2011Blow,Jazar2008Blow,kbag,Qu2014Blow,Soufi2007A,Souplet1,Souplet2} for more explorations on the related problem with homogeneous Neumann boundary condition and nonlocal source.

We next summarize some existing results for problem \eqref{fc1}-\eqref{csz} and then give the purpose of this paper. To this end, let's firstly introduce some functionals, sets and definitions used in these works, which will also be used throughout this paper. We denote by $\|\cdot\|_q$ the standard $L^q(0, a)$ norm for $1\leq q\leq+\infty$. Let
\begin{equation*}
H=\left\{u\in H^2(0, a)\left|\int_0^a u(x)dx=0\right.\right\},
\end{equation*}
then it is well known that $(H, \|\cdot\|)$ with norm $\|u\|=\|u_{xx}\|_2$ is a Banach space. For any $u\in H$, the energy functional, Nehari functional and Nehari manifold associated with problem \eqref{fc1}-\eqref{csz} can be defined by
\begin{equation*}\label{fcx1}
\begin{split}
&J(u)=\frac{1}{2}\|u_{xx}\|_2^2-\frac{1}{p+1}\|u\|_{p+1}^{p+1},\\
&I(u)=\|u_{xx}\|_2^2-\|u\|_{p+1}^{p+1},\\
&\N=\left\{u\in H\ |\ I(u)=0\right\}\setminus\{0\},
 \end{split}
\end{equation*}
then we can define the depth of potential well by $d=\inf_{u\in \N}J(u)$. Let
\begin{equation*}
\N_-=\left\{u\in H\ |\ I(u)<0\right\}.
\end{equation*}
For $k\in\R$, we let $J^{k}=\{u\in H\left|\right.J(u)\leq k\}$, and it is easy to see that for all $\alpha>d$
\begin{equation*}\label{kf}
\begin{split}
\N_\alpha:=\N\cap J^\alpha=\left\{u\in\N\left|\ \|u_{xx}\|_2\leq\sqrt{\frac{2\alpha(p+1)}{p-1}}\right.\right\}\not=\emptyset.
\end{split}\end{equation*}
We further let $\Lambda_\alpha=\sup\left\{\left.\frac{1}{2}\|u\|_2^2\ \right|\ u\in \mathcal {N}_\alpha\right\}$, the conclusion in \cite[Proposition 1.9]{thin} tells us $\Lambda_\alpha\in(0, +\infty)$ for all $\alpha\in (d, +\infty)$. The weak solution of the problem is defined as follows, which can be found in \cite{Qu2016Blow}.
\begin{definition}\label{dy111}
A function $u=u(x,t)$ is called the weak solution of problem \eqref{fc1}-\eqref{csz} on $\Om\times[0, T)$ if $u\in L^{\infty}(0, T; H), u_t\in L^{2}(0, T; H)$ such that
\begin{equation*}\label{rjdy}
  \int_0^t\int_0^a \left[u_t\varphi+u_{xx}\varphi_{xx}-\left(|u|^{p-1}u-\frac{1}{a}\int_{0}^a|u|^{p-1}udx \right)\varphi\right]dxds=0
\end{equation*}
holds for all $\varphi\in L^2(0,T;H^2(0,a))$ with $\varphi_x|_{x=0}=\varphi_x|_{x=a}=0$.
\end{definition}
The local existence, uniqueness and regularity for the weak solutions of problem \eqref{fc1}-\eqref{csz} can be obtained via the standard Galerkin approximation and parabolic theory, see for examples \cite{winkler,sun,zhou}. Moreover, for the maximal existence time $T\in(0, +\infty]$ of the weak solution, if $T=+\infty$, we say the weak solution exists globally; while,
\begin{equation}\label{bpzz}
\mbox{if}\ \ T<+\infty,\ \mbox{then}\ \lim_{t\rightarrow T^-}\|u_{xx}\|_2=+\infty,
\end{equation}
and in this setting, we say the weak solution blows up in finite time.

The authors in \cite{bag,Qu2016Blow,thin,xzm,Zhou2017Blow} studied the properties such as global existence, finite time blow-up and extinction of the weak solution to problem \eqref{fc1}-\eqref{csz}. Roughly speaking, Qu and Zhou \cite{Qu2016Blow} obtained a threshold result for global existence and finite time blow-up of the weak solutions with initial data at low and critical energy levels, i.e., $J(u_0)\leq d$, with the help of the potential well method, the classical Galerkin method and concavity method. Zhou \cite{Zhou2017Blow} estimated the upper bound of blow-up time in the low energy case. Under high initial energy setting, i.e., $J(u_0)>d$, Xu and Zhou motivated by \cite{Gazzola2005Finite} and constructed two sets of initial data such that the corresponding weak solution exists globally or blows up in finite time, respectively \cite{thin}. Recently, by exploiting the boundary conditions and the variational structure of the system, the authors in \cite{xzm} established the exponential decays of weak solutions and energy functional when the weak solution is global. As for the blow-up weak solution, they proved that the weak solution grows exponentially on the maximal time interval and obtain the behavior of energy functional as $t$ tends to the maximal existence time. Moreover, they further discussed the vacuum isolating phenomena of the system. Baghaei \cite{bag} obtained a new blow-up condition on initial energy and given upper and lower bounds for the blow-up time, the exact blow-up time is further obtained under some conditions.

We summarize the finite time blow-up results for problem \eqref{fc1}-\eqref{csz} obtained in \cite{Qu2016Blow,thin} by following corollary. The proofs of items (i) and (ii) can be find in \cite[Theorems 4.2 and 5.2]{Qu2016Blow} and \cite[Theorem 1.11]{thin}, respectively.

\begin{corollary}\label{tuilun}
Assume $I(u_0)<0$ and one of the following two conditions holds
\begin{itemize}
  \item [(i).] $J(u_0)\leq d$;
  \item [(ii).] $J(u_0)>d,$ and
  \begin{equation}\label{ryi}
  \|u_0\|_2^2>2\Lambda_{J(u_0)},
  \end{equation}
\end{itemize}
then the weak solution of problem \eqref{fc1}-\eqref{csz} blows up in finite time.
\end{corollary}

As far as the blow-up phenomenon to problem \eqref{fc1}-\eqref{csz} is concerned, it is nature to ask that whether condition \eqref{ryi} can be removed? In other words, for the weak solution of problem \eqref{fc1}-\eqref{csz} can we have
\begin{equation}\label{ls}
I(u_0)<0 \Rightarrow u(x, t)\ \mbox{blows up in finite time}?
\end{equation}
In fact, for a zero Dirichlet boundary condition, analogous question has been came up in studying the corresponding initial-boundary value problem to the classical semilinear parabolic equation with polynomial nonlinear \cite{Gazzola2005Finite}
\begin{equation}\label{lsss}
\left\{\begin{array}{ll} u_t-\Delta u=|u|^{p-1}u, & (x, t)\in \Omega\times(0, T)\\
u=0, & (x, t)\in\partial\Omega\times(0, T)\\
u(x,0)=u_0(x), & x\in\Omega, \end{array}\right.
\end{equation}
where $\Om$ is an open bounded domain of $\mathbb{R}^n$ with smooth boundary $\partial\Om$. Dickstein et al. in \cite{Dickstein} subsequently proved that the answer is negative and there exist solutions converging to any given steady state, with initial Nehari energy $I(u_0)$ either negative or positive. That is to say, $I(u_0)<0$ is not a sufficient condition for finite time blow-up of solution to \eqref{lsss}. Contrary to the semilinear parabolic equation, Zhu et al. in \cite{zsc} got a positive answer to the corresponding initial-boundary value problem with pseudo-parabolic equation
\begin{equation*}\label{zxl}
u_t-\Delta u_t-\Delta u+u=|u|^{p-2}u.
\end{equation*}
They proved $I(u_0)<0$ is a sufficient condition for finite time blow-up of the weak solution. To a certain extent, this phenomena indicates the difference between the pseudo-parabolic and parabolic equations. Take advantage of the pseudo-parabolic property found in \cite{zsc}, the corresponding positive answers to question \eqref{ls} has been extended to a class of coupled pseudo-parabolic systems \cite{xjzhou} and a nonlocal semilinear pseudo-parabolic equation with conical degeneration \cite{mgxu}. Furthermore, Han \cite{hyz} also proved that only the assumption $I(u_0)<0$ can ensure the corresponding solution blowing up at infinity for a zero Dirichlet initial-boundary problem to semilinear heat equation with logarithmic nonlinearity, i.e., the equation $u_t-\Delta u=u\ln|u|.$

For the weak solution to problem \eqref{fc1}-\eqref{csz}, the goal of present paper is to extend the blow-up results obtained in \cite{Qu2016Blow,thin} via proving that the proposition \eqref{ls} indeed holds. In fact, we find a weaker condition for the initial data such that the solution blows up in finite time. In order to introduce our result more clearly, for the maximal existence time $T\in(0, +\infty]$ of $u$ we set
\begin{equation}\begin{split}\label{sfu}
S^-:=\left\{u_0\in H: \mbox{there is a}\ t_0\in [0, T)\ \mbox{such that}\ I(u(t_0))<0 \right\}.
\end{split}\end{equation}
Now we state our main result.

\begin{theorem}\label{ding5}
Let $u$ be the weak solution of problem \eqref{fc1}-\eqref{csz}. Then $u$ blows up at finite time $T$ in the sense
$$\lim_{t\rightarrow T^-}\|u_{xx}\|_2=+\infty,$$
if and only if $u_0\in S^-$. In other words, $T<+\infty$ if and only if there exists a $t_0\in [0, T)$ such that $I(u(t_0))<0$.
\end{theorem}

\begin{remark}
If $I(u_0)<0$, then $I(u(t_0))<0$ with $t_0=0$, so we can see from the definition of $S^-$ that $u_0\in S^-$. By Theorem \ref{ding5}, there holds
\begin{equation*}
I(u_0)<0\ \Rightarrow\ u_0\in S^- \Leftrightarrow\ \ \mbox{the weak solution blows up in finite time}.
\end{equation*}
Thus, Theorem \ref{ding5} gives a positive answer to \eqref{ls} and then extend the blow-up results obtained in \cite{Qu2016Blow,thin} since our condition only depends on Nehari functional but not on energy functional. In this regard, in addition to the results of \cite{mgxu,xjzhou,zsc} on pseudo-parabolic equation with polynomial nonlinear term, to the best of our knowledge, Theorem \ref{ding5} seems to be the first rigorous finite time blow-up result only depends on $I(u_0)$ and without any assumption on energy functional in relevant parabolic equation with polynomial and nonlocal nonlinearity. This surprising and interesting phenomena is naturally caused by the conservation of mass, such conservation of course is due to the Neumann boundary condition and nonlocal source. Hence, compare with the conclusion in \cite{Dickstein} for problem \eqref{lsss}, our result reflects the important and strong influence of boundary condition and nonlocal source on the solution.
\end{remark}

\begin{remark}
For the question \eqref{ls} to problem \eqref{fc1}-\eqref{csz}, i.e., a fourth-order nonlocal parabolic equation with Neumann boundary condition, as a first step in this direction, our conclusion holds in the one-dimensional case. It is well known that the spatial dimension also affects the dynamic behavior of solutions, thus, whether the conclusion \eqref{ls} remain in force for other spatial dimensions or other related parabolic systems is perhaps an interesting question to be solved.
\end{remark}

Let's briefly summarize the differences of analysis strategies between the systems with pseudo-parabolic mechanism and parabolic structure on this topic. For the both settings, it is widely known that the invariance of the set $\N_-$ is crucial for the proof of finite time blow-up. For the weak solution to problem \eqref{fc1}-\eqref{csz}, the authors of \cite{Qu2016Blow,thin} proved this invariance of $\N_-$ by the means of assumptions on $J(u_0)$ as stated in Corollary \ref{tuilun}. But for the pseudo-parabolic equation cases, one can obtain the corresponding invariance of $\N_-$ without any restriction on $J(u_0)$ due to the special pseudo-parabolic mechanism, which can be done via some close connections between $J(u)$ and $I(u)$, see \cite{mgxu,xjzhou,zsc}. It is worth noting that the analytical method for the pseudo-parabolic equations can not be directly applied to our setting because problem \eqref{fc1}-\eqref{csz} lacks the pseudo-parabolic structure. Hence, we need a new approach.

In this paper, we shall prove Theorem \ref{ding5} by \eqref{zlsh}, the theory of variational principle and a contradiction argument. Roughly speaking, under the assumption $T=+\infty$, we first claim that the global weak solution of problem \eqref{fc1}-\eqref{csz} must satisfy $\liminf_{t\rightarrow+\infty}\|u_{xx}\|_2<+\infty$, see Lemma \ref{hnl}. Then, we can infer from \eqref{zlsh} that $\frac{d}{dt}\|u\|_{p+1}^{p+1}\geq0$ for all $t\in(0, +\infty)$ and $p>1$ via some differential inequality techniques. With this key conclusion in hand, we can get a contradiction for $T=+\infty$ with the help of exploring some properties of $I(u)$, then the conclusion is clear in view of \eqref{bpzz}.

\section{Proof of main result}

We begin this section with some preliminaries. The following lemma tells us the important properties of energy functional and Nehari functional, which comes from the variational structure of problem \eqref{fc1}-\eqref{csz} and have been used widely in \cite{Qu2016Blow,thin,xzm,Zhou2017Blow}, so we here give the conclusions without proving it to avoid repetition, one can refer \cite[Lemma 1]{xzm} for \eqref{nl} and \cite[Lemma 4]{xzm} for \eqref{zzxs23}.

\begin{lemma}\label{ylx1}
For problem \eqref{fc1}-\eqref{csz}, the energy functional $J(u(t))$ is nonincreasing with respect to $t$ and
\begin{equation}\label{nl}
\int_0^t\|u_\tau\|_2^2\ d\tau+J(u(t))=J(u_0).
\end{equation}
Moreover, for the Nehair functional there holds
\begin{equation}\label{zzxs23}
\frac{d}{dt}\|u\|_2^2=-2I(u).
\end{equation}
\end{lemma}

\begin{lemma}\label{hnl}
Let $u$ be the global weak solution to problem \eqref{fc1}-\eqref{csz} over $(0, +\infty)$. Then it has the property
\begin{equation}\label{fc6}
\liminf_{t\rightarrow+\infty}\|u_{xx}\|_2<+\infty.
\end{equation}
\end{lemma}

\begin{proof}
For contradiction, we assume
\begin{equation}\label{cfd2}
\lim_{t\rightarrow+\infty}\|u_{xx}\|_2=+\infty.
\end{equation}
Let $M(t)=\frac{1}{2}\int_0^t\|u\|_2^2\ d\tau$. Then it is easy to see $M'(t)=\frac{1}{2}\|u\|_2^2$, by \eqref{zzxs23}, the definitions of $J(u), I(u)$ and \eqref{nl} we further have
\begin{equation}\label{eeg}\begin{split}
M''(t)=-I(u)&=-(p+1)J(u)+\frac{p-1}{2}\|u_{xx}\|_2^2\\
&\geq-(p+1)J(u_0)+\frac{p-1}{2}\|u_{xx}\|_2^2,
\end{split}\end{equation}
which together with \eqref{cfd2} imply that $\lim_{t\rightarrow+\infty}M''(t)=+\infty$, this further tells us
\begin{equation}\label{gfhf}
\lim_{t\rightarrow+\infty}M(t)=\lim_{t\rightarrow+\infty}M'(t)=+\infty.
\end{equation}
By \eqref{nl} and the second equality in \eqref{eeg} we know
\begin{equation*}
M''(t)=(p+1)\int_0^t\|u_\tau\|_2^2\ d\tau-(p+1)J(u_0)+\frac{p-1}{2}\|u_{xx}\|_2^2.
\end{equation*}
With the help of \eqref{cfd2} we can see there exists suitable lager $\tau_1$ such that $\frac{p-1}{2}\|u_{xx}\|_2^2-(p+1)J(u_0)\geq0$ for all $t\in(\tau_1, +\infty)$, so
\begin{equation*}
M''(t)\geq (p+1)\int_0^t\|u_\tau\|_2^2\ d\tau,\quad \forall t\in(\tau_1, +\infty).
\end{equation*}
By multiplying this inequality by $M(t)$ and using the definitions of $M(t), M'(t)$ as well as the H\"{o}lder inequality, we find for all $t\in(\tau_1, +\infty)$ that
\begin{equation*}\begin{split}
M''(t)M(t)&\geq\frac{p+1}{2}\int_0^t\|u_\tau\|_2^2\ d\tau\int_0^t\|u\|_2^2\ d\tau\\
&\geq\frac{p+1}{8}\left[\int_0^t\left(\frac{d}{dt}\int_0^a u^2dx\right)d\tau\right]^2\\
&=\frac{p+1}{2}\left(M'(t)-M'(0)\right)^2.
\end{split}\end{equation*}
By \eqref{gfhf} we can find suitable lager $\tau_2$ such that for all $t\in(\tau_2, +\infty)$ there holds $M'(t)>\frac{1}{\varepsilon}M'(0)$, where $\varepsilon$ is a constant satisfies
\begin{equation}\label{ebxn}
\varepsilon\in\left(0,\ 1-\sqrt{\frac{2}{p+1}}\right).
\end{equation}
Hence, let $\tau=\max\{\tau_1, \tau_2\}$, for all $t\in(\tau, +\infty)$ there holds
\begin{equation*}\begin{split}
M''(t)M(t)>\frac{p+1}{2}(1-\varepsilon)^2\left[M'(t)\right]^2.
\end{split}\end{equation*}
It is easy to see that $\frac{p+1}{2}(1-\varepsilon)^2>1$ due to \eqref{ebxn}. Then $\eta=\frac{(p+1)(1-\varepsilon)^2-2}{2}>0$. Let $F(t)=[M(t)]^{-\eta}$ for all $t\in(0, +\infty)$. By the definition of $M(t)$ we know
\begin{equation*}
F(t)>0\ \ \ \mbox{and}\ \ \ F'(t)=-\eta[M(t)]^{-(\eta+1)}M'(t)<0
\end{equation*}
for all $t\in(0, +\infty)$ and
\begin{equation*}
F''(t)=-\eta [M(t)]^{-(\eta+2)}\left\{M''(t)M(t)-\frac{p+1}{2}(1-\varepsilon)^2\left[M'(t)\right]^2\right\}<0,\quad\forall t\in[\tau, +\infty).
\end{equation*}
This implies that $F(t)$ reaches 0 in a finite time, which contradicts the fact $F(t)>0$ for all $t\in(0, +\infty)$. Thus, the global weak solution to problem \eqref{fc1}-\eqref{csz} does possesses the property \eqref{fc6}.
\end{proof}

Now, we can prove our main result.

\begin{proof}[Proof of Theorem \ref{ding5}]
\textbf{Sufficiency.}\ \ we first prove the sufficiency, i.e.,
\begin{equation*}
u_0\in S^-\Rightarrow T<+\infty.
\end{equation*}
Arguing with contradiction, suppose that $u$ exist globally, i.e., $T=+\infty$. Then for any finite time interval, $\|u_{xx}\|_2$ must exist finitely, this together with Lemma \ref{hnl} suggest that for any $t\in(0, +\infty]$, we have $\|u_{xx}\|_2<+\infty$. Thus, the Sobolev embedding theorem $H\hookrightarrow L^\infty(0, a)$ further tells us
\begin{equation}\label{fc66}
\|u\|_{\infty}<+\infty, \quad \forall t\in(0, +\infty].
\end{equation}
Hence, we can see from \eqref{zlsh} that
\begin{equation*}\begin{split}
\frac{d}{dt}\|u\|_{p+1}^{p+1}=(p+1)\int_0^a u^p u_t\ dx\geq -(p+1)\|u\|^p_{\infty}\left|\int_0^a u_t\ dx\right|=0.
\end{split}\end{equation*}
Then, by the definitions of $J(u(t)), I(u(t))$ and \eqref{nl} we have that for all $t\in (0, +\infty)$,
\begin{equation*}\begin{split}
\frac{d}{dt}I(u(t))=\frac{d}{dt}\left[2J(u(t))-\frac{p-1}{p+1}\|u(t)\|_{p+1}^{p+1}\right]\leq-2\|u_t\|_{2}^2.
\end{split}\end{equation*}
Since $u_0\in S^-$, by the definition of $S^-$ as given in \eqref{sfu} we know $I(u(t_0))<0$ for some $t_0\in[0, +\infty)$. It follows from integrating above inequality over $(t_0, +\infty)$ that
\begin{equation}\label{5.2}
I(u(t))\leq I(u(t_0))-2\int_{t_0}^{t}\|u_\tau\|_{2}^2\ d\tau,\ \ \forall t\in(t_0, +\infty).
\end{equation}
If $J(u(t_0))\leq d$, then we can combine it with $I(u(t_0))<0$ and take $t_0$ as initial time to deduce that $u$ blows up at finite time $T$ due to Corollary \ref{tuilun} (i). Therefore, to finish the proof of this theorem, we need only consider the case that
\begin{equation*}
d<J(u(t))\leq J(u_0),\quad\forall t\in [0, +\infty).
\end{equation*}
Above inequalities and Lemma \ref{ylx1} suggest that $J(u(t))$ is non-increasing and bounded on $[0, +\infty)$, so the limit $\lim_{t\rightarrow+\infty}J(u(t))=L$ exists, which combines with above inequalities suggest that $d\leq L\leq J(u_0)$. Then let $t\rightarrow+\infty$ in Lemma \ref{ylx1}, we can see
\begin{equation}\label{xuyao}
\int_{0}^{+\infty}\|u_t\|^2_{2}\ dt=J(u_0)-L\geq0,
\end{equation}
which further deduces that there exists a diverging sequence $\{t_n\}$ such that
\begin{equation*}
\lim_{n\rightarrow+\infty}\|u_t(t_n)\|^2_{2}=0.
\end{equation*}
For such diverging sequence $\{t_n\}$, by \eqref{zzxs23}, the H\"{o}lder inequality and \eqref{fc66} we know
\begin{equation*}\begin{split}
\left|I(u(t_n))\right|=\left|\int_0^a u(t_n)u_t(t_n)\ dx \right|\leq\|u_t(t_n)\|_{2}\ \|u(t_n)\|_{2}\rightarrow 0,\ \ \mbox{as}\ \ n\rightarrow+\infty.
\end{split}\end{equation*}

However, by \eqref{5.2}, $I(u(t_0))<0$ and \eqref{xuyao} we know as $n\rightarrow+\infty$ there holds
\begin{equation*}\begin{split}
|I(u(t_n))|&\geq -I(u(t_0))+2\int_{t_0}^{t_n}\|u_t\|_{2}^2\ dt\\
& \rightarrow -I(u(t_0))+2\int_{t_0}^{+\infty}\|u_t\|_{2}^2\ dt\geq -I(u(t_0))>0.
\end{split}\end{equation*}
So a contradiction occurs. Therefore, $T<+\infty$ and $u(x, t)$ blows up in finite time in the sense of \eqref{bpzz}.

\textbf{Necessity.}\ \ Next, we aim to obtain the necessity, i.e.,
\begin{equation*}
T<+\infty \Rightarrow u_0\in S^-.
\end{equation*}
In this end, we will prove the following equivalent proposition:
\begin{equation*}\label{djmt}
u_0\not\in S^-\ \Rightarrow\ T=+\infty.
\end{equation*}
By $u_0\not\in S^-$ and the definition of $S^-$ as given in \eqref{sfu} we know $I(u(t))\geq0$ for all $t\in[0, T)$. By the fact that $J(u(t))\leq J(u_0)$ and the definitions of $J(u)$ and $I(u)$, it holds that
\begin{equation*}\begin{split}
\frac{p-1}{2(p+1)}\|u_{xx}\|^2_{2}&\leq\frac{p-1}{2(p+1)}\|u_{xx}\|^2_{2}+\frac{1}{p+1}I(u(t))=J(u(t))\leq J(u_0).
\end{split}\end{equation*}
This implies that $\|u_{xx}\|_{2}$ is uniformly bounded on $[0, T)$, so \eqref{bpzz} implies $T=+\infty$ and $u$ exists globally.
\end{proof}

\section*{Acknowledgement}
The first author is supported by the Scientific Research Fund of Zhejiang Provincial Education Department (Grant No. Y202353398). The second author is supported by the National Natural Science Foundation of China (Grant No. 12301261), and Sichuan Science and Technology Program(Grant No. 2023NSFSC1365), and the Scientific Research Starting Project of SWPU (Grant No. 2021QHZ016). The third (corresponding) author is supported by National Natural Science Foundation of China (Grant No. 12201567), and Scientific Research Fund of Zhejiang Provincial Education Department (Grant No. Y202249802), and the Young Doctor Program of Zhejiang Normal University (Grant No. 2021ZS0802). The fourth author is supported by the Science and Technology Research Program of Chongqing Municipal Education Commission (Grant No. KJQN202300808), and Scientific Research Foundation of Chongqing Technology and Business University (Grant No. 2156020).



\begin{thebibliography}{10}

\bibitem{bag}
K. Baghaei,
\newblock Blow-up, non-extinction and exponential growth of solutions to a fourth-order parabolic equation,
\newblock {\em C. R. Math. Acad. Sci. Paris}, 350: 47--56, 2022.

\bibitem{Budd1994Blow}
C. Budd, J. Dold and A. Stuart,
\newblock Blow-up in a system of partial differential equations with conserved
  first integral. part II: Problems with convection,
\newblock {\em SIAM J. Appl. Math.}, 54(3): 18--25, 1994.

\bibitem{Dickstein}
F. Dickstein, N. Mizoguchi, P. Souplet and F. Weissler,
\newblock Transversality of stable and Nehari manifolds for a semilinear heat equation,
\newblock {\em Calc. Var. Partial Differ. Equ.} 42: 547--562, 2011.

\bibitem{Gao2011Blow}
W. Gao and Y. Han,
\newblock Blow-up of a nonlocal semilinear parabolic equation with positive
  initial energy,
\newblock {\em Appl. Math. Lett.}, 24(5): 784--788, 2011.

\bibitem{Gazzola2005Finite}
F. Gazzola and T. Weth,
\newblock Finite time blow-up and global solutions for semilinear parabolic
  equations with initial data at high energy level,
\newblock {\em Diff. Integral Equ.}, 18(9): 961--990, 2005.

\bibitem{hyz}
Y. Han,
\newblock Blow-up at infinity of solutions to a semilinear heat equation with logarithmic nonlinearity,
\newblock {\em J. Math. Anal. Appl.}, 474: 513--517, 2019.

\bibitem{Jazar2008Blow}
M.~Jazar and R.~Kiwan,
\newblock Blow-up of a non-local semilinear parabolic equation with Neumann
  boundary conditions,
\newblock {\em Ann. Inst. H. Poincar\'{e} C Anal. Non Lin\'{e}aire}, 25(2): 215--218, 2008.

\bibitem{kbag}
A. Khelghati and K. Baghaei,
\newblock Blow-up phenomena for a nonlocal semilinear parabolic equation with positive initial energy,
\newblock {\em Comput. Math. Appl.}, 70: 896--902, 2015.

\bibitem{winkler}
B. King, O. Stein and M. Winkler,
\newblock A fourth-order parabolic equation modeling epitaxial thin film growth,
\newblock {\em J. Math. Anal. Appl.}, 286(2): 459--490, 2003.

\bibitem{mgxu}
J. Meng and G. Xu,
\newblock Higher regularity and finite time blow-up to nonlocal pseudo-parabolic equation with conical degeneration,
\newblock {\em Appl. Math. Optim.} 88(2): 50, 2023.

\bibitem{ors}
M. Ortiz, E. Repetto and H. Si,
\newblock A continuum model of kinetic roughening and coarsening in thin films,
\newblock {\em J. Meth. Phys. Solids}, 47(4): 697--730, 1999.

\bibitem{Qu2014Blow}
C. Qu, X. Bai and S. Zheng,
\newblock Blow-up versus extinction in a nonlocal p-Laplace equation with
  Neumann boundary conditions,
\newblock {\em J. Math. Anal. Appl.}, 412(1): 326--333, 2014.

\bibitem{Qu2016Blow}
C. Qu and W. Zhou,
\newblock Blow-up and extinction for a thin-film equation with initial-boundary
  value conditions,
\newblock {\em J. Math. Anal. Appl.},
  436(2): 796--809, 2016.

\bibitem{Soufi2007A}
A. Soufi, M.~Jazar and R.~Monneau,
\newblock A gamma-convergence argument for the blow-up of a non-local
  semilinear parabolic equation with Neumann boundary conditions,
\newblock {\em Ann. Inst. H. Poincar\'{e} C Anal. Non Lin\'{e}aire}, 24(1): 17--39, 2007.

\bibitem{Souplet1}
P. Souplet,
\newblock  Blow-up in nonlocal reaction-diffusion equations,
\newblock {\em SIAM J. Math. Anal.}, 29 (6), 1301--1334, 1998.

\bibitem{Souplet2}
P. Souplet,
\newblock Uniform blow-up profiles and boundary behavior for diffusion equations with nonlocal nonlinear source,
\newblock {\em J. Differential Equations}, 153: 374--406, 1999.

\bibitem{sun}
F. Sun, L. Liu and Y. Wu,
\newblock Finite time blow-up for a thin-film equation with initial data at arbitrary energy level,
\newblock {\em J. Math. Anal. Appl.}, 458: 9--20, 2018.

\bibitem{thin}
G. Xu and J. Zhou,
\newblock Global existence and finite time blow-up of the solution for a thin-film equation with high initial energy,
\newblock {\em J. Math. Anal. Appl.}, 458: 521--535, 2018.

\bibitem{xzm}
G. Xu, J. Zhou and C. Mu,
\newblock Global existence, finite time blow-up, and vacuum isolating phenomenon for a class of thin-film equation,
\newblock {\em J. Dyn. Control Syst.}, 26(2): 265-288, 2020.

\bibitem{xjzhou}
G. Xu and J. Zhou,
\newblock Sufficient and necessary condition for the blowing-up solution to a class of coupled pseudo-parabolic equations,
\newblock {\em Appl. Math. Lett.} 128: 107886, 2022.

\bibitem{za}
A. Zangwill,
\newblock Some causes and a consequence of epitaxial roughening,
\newblock {\em J. Cryst. Growth}, 163(1): 8-21, 1996.

\bibitem{Zhou2017Blow}
J. Zhou,
\newblock Blow-up for a thin-film equation with positive initial energy,
\newblock {\em J. Math. Anal. Appl.}, 446(1): 1133--1138, 2017.

\bibitem{zhou}
J. Zhou,
\newblock Global asymptotical behavior and some new blow-up conditions of solutions to a thin-film equation,
\newblock {\em J. Math. Anal. Appl.}, 464: 1290--1312, 2018.


\bibitem{zsc}
X. Zhu, F. Li and Y. Li,
\newblock Some sharp results about the global existence and blowup of solutions to a class of pseudo-parabolic equations,
\newblock {\em Proc. Roy. Soc. Edinburgh Sect. A}, 147A, 1311--1331, 2017.

\end{thebibliography}

\end{document}